\documentclass{amsart}

\usepackage[english]{babel}
\usepackage[latin1]{inputenc}
\usepackage{amsmath}
\usepackage{amsthm}
\usepackage{amssymb}
\usepackage{tikz}

\usepackage{imakeidx}
\usepackage{appendix}
\usepackage{tikz-cd}
\usepackage{verbatim}
\usepackage{enumitem}
\usepackage{multicol}
\usepackage{adjustbox}

\usepackage{hyperref}
\usepackage{cleveref}

\hypersetup{colorlinks=true,linkcolor=blue,anchorcolor=blue,citecolor=blue}

\usepackage{todonotes}

\usepackage{mathtools}

\newtheorem{thm}{Theorem}[section]
\newtheorem{prop}[thm]{Proposition}
\newtheorem{lemma}[thm]{Lemma}

\newtheorem{question}[thm]{Question}

\theoremstyle{definition}
\newtheorem{defin}[thm]{Definition}

\theoremstyle{remark}

\newtheorem{rmk}[thm]{Remark}
\numberwithin{equation}{section}

\newcommand{\Q}{\mathbb Q}
\newcommand{\F}{\mathbb F}
\newcommand{\C}{\mathbb C}
\newcommand{\Z}{\mathbb Z}
\newcommand{\G}{\mathbb G}

\renewcommand{\c}{\subseteq}

\newcommand{\mc}[1]{\mathcal{#1}}
\newcommand{\cl}{\overline}
\newcommand{\set}[1]{\left\{#1\right\}}
\renewcommand{\phi}{\varphi}

\newcommand{\on}[1]{\operatorname{#1}}
\newcommand{\ang}[1]{\left \langle{#1}\right \rangle}

\newcommand{\ext}{\mbox{$\bigwedge$}}

\title{Torsion classes in the equivariant Chow groups of algebraic tori}
\author{Federico Scavia}

\address{Department of Mathematics\\
	University of British Columbia\\
	Vancouver, BC V6T 1Z2\\Canada}
\email{scavia@math.ubc.ca}
\thanks{The author was partially supported by a graduate fellowship from the University of British Columbia. This research was made possible through funding provided by Mitacs.}



\begin{document}
	\begin{abstract}
		We give an example of an algebraic torus $T$ such that the group $\on{CH}^2(BT)_{\on{tors}}$ is non-trivial. This answers a question of Blinstein and Merkurjev.
	\end{abstract}
	
	\maketitle

	\section{Introduction}	
	Let $F$ be a field, and let $G$ be a linear algebraic group over $F$. Let $i\geq 0$ be an integer, let $V$ be a linear representation of $G$ over $F$, and assume that there exists a $G$-invariant open subscheme $U$ of $V$ such that $U$ is the total space of a $G$-torsor $U\to U/G$ and $V\setminus U$ has codimension at least $i+1$ in $V$. Following B. Totaro \cite[Definition 1.2]{totaro1999chow}, we define \[\on{CH}^i(BG):=\on{CH}^i(U/G).\] This definition does not depend on the choice of $V$ and $U$; see \cite[Theorem 1.1]{totaro1999chow}. The graded abelian group $\on{CH}^*(BG):=\oplus_{i\geq 0}\on{CH}^i(BG)$ has the structure of a commutative ring with identity.
	
	If $T$ is a split $F$-torus, and $\hat{T}$ is the character lattice of $T$, then there is a canonical isomorphism $\on{Sym}(\hat{T})\simeq \on{CH}^*(BT)$. Thus, if $T$ has rank $n$, $\on{CH}^*(BT)$ is a polynomial ring with $n$ generators in degree $1$, and in particular its underlying additive group is torsion-free.
	
	When $G$ is a finite group, a lot of work on $\on{CH}^*(BG)$ has been carried out by a number of authors, for example N. Yagita \cite{yagita2012chow}, P. Guillot \cite{guillot2005steenrod} and Totaro. Totaro's book \cite{totaro2014group} is devoted to the study of $\on{CH}^*(BG)$ and to its relation to the group cohomology of $G$. 
	
	When $G$ is a split reductive group, there is an extensive literature dealing with computations of $\on{CH}^*(BG)$.
	For instance, the ring $\on{CH}^*(BG)$ has been computed for $G=\on{GL}_n,\on{SL}_n,\on{Sp}_{2n}$ by Totaro \cite{totaro1999chow}, for $G=\on{O}_n,\on{SO}_{2n+1}$ by Totaro and R. Pandharipande \cite{totaro1999chow} \cite{pandharipande1998equivariant}, for $G=\on{SO}_{2n}$ by R. Field \cite{field2012chow}, for $\on{G}_2$ by N. Yagita \cite{yagita2005applications}, for $\on{PGL}_3$ by G. Vezzosi \cite{vezzosi2000chow}, and for $\on{PGL}_p$ (additively) by A. Vistoli \cite{vistoli2007cohomology}. 
	
	Let $F_s$ be a separable closure of $F$, let $\mc{G}:=\on{Gal}(F_s/F)$ be the absolute Galois group of $F$. If $X$ is an $F$-scheme, we define $X_s:=X\times_FF_s$. When $G$ is not assumed to be split, a lot less is known about $\on{CH}^*(BG)$. Assume that $G=T$ is an $F$-torus, not necessarily split. Then we have canonical isomorphisms \[\on{CH}^1(BT)\simeq \on{CH}^1(BT_s)^{\mc{G}}\simeq (\hat{T_s})^{\mc{G}}.\]
	On the other hand, the natural homomorphism \[\on{CH}^2(BT)\to \on{CH}^2(BT_s)^{\mc{G}}\] is not surjective in general; many examples can be obtained from \cite[Lemma 4.2, Theorem 4.10, Theorem 4.13]{blinstein2013cohomological}.
	
	When $X$ is a smooth variety over $F$, the natural map \[\on{CH}^2(X)\to \on{CH}^2(X_s)^{\mc{G}}\] is in general neither injective nor surjective, that is, Galois descent for codimension $2$ cycles may fail. It is a difficult and interesting problem to study the kernel and cokernel of the previous map, even for special families of varieties $X$. An extensive literature is devoted to it; here we limit ourselves to mention the works of A. Pirutka \cite{pirutka2011groupe}, J.-L. Colliot-Th{\'e}l{\`e}ne \cite{colliot2013cycles}, Colliot-Th{\'e}l{\`e}ne and B. Kahn \cite{colliot2015descente}, and R. Parimala and V. Suresh \cite{parimala2016degree}.

	Since $\on{CH}^2(BT_s)$ is torsion-free, a norm argument shows that \[\on{Ker}(\on{CH}^2(BT)\to \on{CH}^2(BT_s)^{\mc{G}})=\on{CH}^2(BT)_{\on{tors}},\] where $\on{CH}^2(BT)_{\on{tors}}$ is the torsion subgroup of $\on{CH}^2(BT)$. The group $\on{CH}^2(BT)_{\on{tors}}$ plays a prominent role in work of S.~Blinstein and A.~Merkurjev, where it appears as the first term of the exact sequence of \cite[Theorem B]{blinstein2013cohomological}. In \cite[Theorem 4.7]{blinstein2013cohomological}, Blinstein and Merkurjev showed that 
	$\on{CH}^2(BT)_{\on{tors}}$ is finite and $2\cdot\on{CH}^2(BT)_{\on{tors}}=0$. They posed the following question.
	
	\begin{question}\label{question} {\rm (}\cite[Question 4.9]{blinstein2013cohomological}{\rm)}
		Is $\on{CH}^2(BT)_{\on{tors}}$ trivial for every torus $T$?
	\end{question}
	
	Merkurjev studied this question further in~\cite{merkurjev2019pairing}.
	He showed that $\on{CH}^2(BT)_{\on{tors}}=0$ in many cases, for example:
	\begin{itemize}
		\item[--] when $BT$ is $2$-retract rational, by \cite[Corollary 5.5]{merkurjev2019pairing};
		\item[--] when the $2$-Sylow subgroups of the splitting group of $T$ are cyclic or Klein four-groups, by \cite[Proposition 2.1(2), Example 4.3, and Corollary 5.3]{merkurjev2019pairing};
		\item[--] when $\on{char}F=2$, by \cite[Corollary 5.5]{merkurjev2019pairing};
		\item[--] when $T=R_{E/F}(\G_{\on{m}})/\G_{\on{m}}$ and $E/F$ is a finite Galois extension, by \cite[Example 4.2, Corollary 5.3]{merkurjev2019pairing}. 
	\end{itemize}

	The purpose of this paper is to show that \Cref{question} has a negative answer. 
	
	\begin{thm}\label{mainthm}
		There exist a field $F$ and an $F$-torus $T$ such that $\on{CH}^2(BT)_{\on{tors}}$ is not trivial.
	\end{thm}
	
	In our example, the splitting group $G$ of $T$ is a $2$-Sylow subgroup of the Suzuki group $\on{Sz}(8)$, and $F=\Q(V)^G$, where $V$ is a faithful representation of $G$ over $\Q$. The group $G$ has order $64$; no counterexample with a splitting group of smaller order can
	be detected using our method. The torus $T$ has dimension $2^{12}-2^6+1=4033$.
	
	The paper is structured as follows. In \Cref{sec2}, we recall a construction due to Merkurjev \cite{merkurjev2019pairing}, which to every $G$-lattice $L$ associates an abelian group $\Phi(G,L)$. By a result of Merkurjev, to show that \Cref{question} has a negative answer, it suffices to exhibit $G$ and $L$ such that $\Phi(G,L)\neq 0$; see \Cref{merk-thm}. This reduces \Cref{question} to a problem in integral representation theory. In \Cref{sec3}, we associate to every finite group $G$ a $G$-lattice $M$. In Sections \ref{sec4} and \ref{sec5} we show that if the group cohomology of $G$ with $\Z/2$ coefficients satisfies a certain condition, then $\Phi(G,M)\neq 0$; see \Cref{condition}(b). Finally, in \Cref{sec6}, we show that the condition of \Cref{condition}(b) is satisfied when $G$ is a $2$-Sylow subgroup of $\on{Sz}(8)$.
	
	\section{Merkurjev's reformulation of Question 1.1}\label{sec2}
	
	Let $G$ be a finite group, and let $L$ be a $G$-lattice, that is, a finitely generated free $G$-module. By definition, the second exterior power $\ext^2(L)$ of $L$ is the quotient of $L\otimes L$ by the subgroup generated by all elements of the form $x\otimes x$, $x\in L$. We denote by $\Gamma^2(L)$ the factor group of $L\otimes L$ by the subgroup generated by $x\otimes y+y\otimes x$, $x,y\in L$. We write $x\wedge y$ for the coset of $x\otimes y$ in $\ext^2(L)$, and $x\star y$ for the coset of $x\otimes y$ in $\Gamma^2(L)$.
	
	We have a short exact sequence 
	\begin{equation}\label{exterior}
	0\to L/2\xrightarrow{\iota} \Gamma^2(L)\xrightarrow{\pi} \ext^2(L)\to 0,
	\end{equation}
	where $\iota(x+2L)=x\star x$, and $\pi(x\star y)=x\wedge y$.
	We write \[\alpha_L:H^1(G,\ext^2(L))\to H^2(G,L/2)\] for the connecting homomorphism for (\ref{exterior}). Recall that a $G$-lattice is called a permutation lattice if it admits a permutation basis, i.e., a $\Z$-basis stable under the $G$-action. A $G$-lattice $L'$ is said to be stably equivalent to $L$ if there exist permutation $G$-lattices $P$ and $P'$ such that $L\oplus P\simeq L'\oplus P'$.
	
	\begin{lemma}\label{setup}
		\begin{enumerate}[label=(\alph*)]
			\item Assume that $L$ is a  permutation $G$-lattice, and let $x_1,\dots,x_n$ be a permutation basis of $L$. Then the homomorphism \[\Gamma^2(L)\to L/2,\qquad x_i\star x_j\mapsto 0\ (i\neq j), \quad x_i\star x_i\mapsto x_i+2L,\] defines a splitting of (\ref{exterior}). Moreover, the homomorphism
			\[\ext^2(L)\to \Gamma^2(L),\qquad x_i\wedge x_j\mapsto x_i\star x_j\ (i<j)\] is a section of $\pi$.
			\item Let $L'$ be a $G$-lattice stably equivalent to $L$. Then $\on{Im}(\alpha_L)\simeq \on{Im}(\alpha_{L'})$.
		\end{enumerate}
	\end{lemma}
	\begin{proof}
		This is contained in  \cite[\S 2]{merkurjev2019pairing}.	
	\end{proof}
	
	\begin{rmk}\label{ordering}
		In \Cref{setup}(a), the homomorphism $\Gamma^2(L)\to L/2$ is clearly independent of the ordering of the $x_i$. Since $x\wedge y=-y\wedge x$ and $x\star y=-y\star x$ for every $x,y\in L$, the homomorphism $\ext^2(L)\to \Gamma^2(L)$ sends $x_i\wedge x_j$ to $x_i\star x_j$ for every $i\neq j$. In particular, the homomorphism $\ext^2(L)\to \Gamma^2(L)$ is also independent of the choice of the ordering of the $x_i$. 
	\end{rmk}
	
	Recall that a coflasque resolution of $L$ is a short exact sequence of $G$-lattices \begin{equation}\label{l'} 0\to L''\to L'\to L\to 0\end{equation} such that $L'$ is a permutation $G$-lattice and $L''$ is coflasque, i.e., $H^1(H,L'')=0$ for every subgroup $H$ of $G$. By \cite[Lemme 3]{colliot1977r}, coflasque resolutions always exist. Let (\ref{l'}) be a coflasque resolution of $L$, and define
	\[\Phi(G,L):=\on{Im}(\alpha_{L''}).\] By \Cref{setup}(b) and \cite[Lemme 5]{colliot1977r}, this definition does not depend on the choice of a coflasque resolution of $L$.
	
	The following is a reformulation of \Cref{question} purely in terms of integral representation theory, and is the starting point for the present work. It is due to Merkurjev \cite{merkurjev2019pairing}, and builds upon the results of Blinstein-Merkurjev \cite{blinstein2013cohomological}.
	
	\begin{thm}\label{merk-thm}
		Let $F$ be a field, let $T$ be an $F$-torus with character lattice $\hat{T}$, minimal splitting field $E$ and splitting group $G=\on{Gal}(E/F)$.
		\begin{enumerate}[label=(\alph*)]
			\item The group $\on{CH}^2(BT)_{\on{tors}}$ is a quotient of $\Phi(G,\hat{T})$. 
			\item Let $V$ be a faithful representation of $G$ over $\Q$, and assume that $E=\Q(V)$ and $F=\Q(V)^G$. Then $\Phi(G,\hat{T})\simeq \on{CH}^2(BT)_{\on{tors}}$.
		\end{enumerate}
	\end{thm}
	\begin{proof}
		(a) This is \cite[Corollary 5.2]{merkurjev2019pairing}.
		
		(b) Let $T^{\circ}$ be the dual torus of $T$. For every field extension $K/F$, there is a pairing
		\begin{equation}\label{pairing}H^1(K,T^{\circ})\times H^1(K,T)\xrightarrow{\cup} H^3(K,\Q/\Z(2)),\qquad (\alpha,\beta)\mapsto \alpha\cup\beta;\end{equation} see \cite[(4-5)]{blinstein2013cohomological}.
		By definition, the left kernel of (\ref{pairing}) is the subgroup of all $\alpha\in H^1(F,T^{\circ})$ such that for every field extension $K/F$ and every $\beta\in H^1(K,T)$ we have $\alpha_K\cup\beta=0$.  
		
		By \cite[Proposition 5.6]{merkurjev2019pairing}, the left kernel of the pairing (\ref{pairing}) is isomorphic to $\Phi(G,\hat{T})$. By \cite[Theorem B]{blinstein2013cohomological}, the left kernel of (\ref{pairing}) is isomorphic to $\on{CH}^2(BT)_{\on{tors}}$.
	\end{proof}

	\section{The example}\label{sec3}
	Let $G$ be a finite group. In this section we define a $G$-lattice $M$. In \Cref{condition}, we will show that $\Phi(G,M)\neq 0$ when the group cohomology of $G$ with $\Z/2$ coefficients satisfies a certain condition.
	
	The $G$-lattice $\Z[G\times G]$ is a free $\Z[G]$-module, and has a canonical permutation basis $\set{(g,g'):g,g'\in G}$. The free $G$-lattice $\Z[G]\oplus \Z[G]$ has a canonical basis $\set{(g,0),(0,g):g\in G}$. Let $\gamma:=\sum_{g\in G} g \Z[G]$. Consider the homomorphism \[\rho:\Z[G]^{\oplus 2}\to \Z[G\times G],\qquad (g,0)\mapsto \sum_{g'\in G}(g,g'),\quad (0,g)\mapsto \sum_{g'\in G}(g',g).\] Then $\on{Ker}\rho=\ang{(\gamma,-\gamma)}$ has rank $1$ and trivial $G$-action. Define \[M:=\on{Coker}\rho,\] and let $\pi: \Z[G\times G]\to M$ be the natural projection. 
	
	\begin{lemma}
		The $G$-module $M$ is $\Z$-free.
	\end{lemma}
	
	\begin{proof}
		Let \[x=\sum_{g,g'\in G}a_{g,g'}(g,g')\in \Z[G\times G].\] It is easy to show that $x\in \on{Im}\rho$ if and only if there exist integers $b_g,c_g$ (where $g\in G$) such that $a_{g,g'}=b_g+c_{g'}$ for every $g,g'\in G$.
		
		Assume that $n\pi(x)=0$ for some $n\geq 1$. This means that $nx\in \on{Im}\rho$. Thus, for every $g,g'\in G$, we can find integers $p_g,q_{g'}$ such that $na_{g,g'}=p_g+q_{g'}$. We may write $p_g=nb_g+p'_g$, $q_{g'}=nc_{g'}-q'_{g'}$ where $0\leq p'_g<n$ and $0\leq q'_{g'}<n$ for all $g,g'\in G$. Then $n(a_{g,g'}-b_g-c_{g'})=p'_g-q'_{g'}$. In particular, $n$ divides $p'_g-q'_{g'}$. On the other hand, we have $|p'_g-q'_{g'}|<n$, hence $p'_g=q'_{g'}$. We conclude that $a_{g,g'}=b_g+c_{g'}$ for every $g,g'\in G$, hence $x\in \on{Im}\rho$. This shows that $M$ is torsion-free, as desired.
	\end{proof}
	
	By construction, we have a long exact sequence \begin{equation}\label{cofl0}0\to \Z\xrightarrow{i} \Z[G]^{\oplus 2}\xrightarrow{\rho} \Z[G\times G]\xrightarrow{\pi} M\to 0,\end{equation} where $i(1):=(\gamma,-\gamma)$. We set $N:=\on{Im}\rho$, and we split (\ref{cofl0}) into the two short exact sequences \begin{equation}\label{cofl0a}0\to N\to \Z[G\times G]\xrightarrow{\pi} M\to 0\end{equation} and \begin{equation}\label{cofl0b}0\to \Z\xrightarrow{i} \Z[G]^{\oplus 2}\to N\to 0.\end{equation}
	Let \begin{equation}\label{cofl-p}0\to R\to P\xrightarrow{\pi'} M\to 0\end{equation} be a coflasque resolution of $M$. We define a $G$-lattice $Q$ and homomorphisms $\iota_1:Q\to \Z[G\times G]$ and $\iota_2:Q\to P$ by the short exact sequence 
	\begin{equation}\label{cofl2}0\to Q\xrightarrow{(\iota_1,-\iota_2)} \Z[G\times G]\oplus P\xrightarrow{\pi+\pi'}M\to 0.
	\end{equation} 
	Thus we have a commutative diagram
	\[
	\begin{tikzcd}
	&& 0 & 0\\
	0 \arrow[r] & N \arrow[r] & \Z[G\times G] \arrow[r,"\pi"]\arrow[u] & M \arrow[r]\arrow[u] & 0  \\
	0 \arrow[r]   & N \arrow[r] \arrow[u,equal] & Q \arrow[r,"\iota_2"] \arrow[u,"\iota_1"] & P \arrow[r]\arrow[u,"\pi'"] & 0  \\
	&& R\arrow[u] \arrow[r,equal] & R\arrow[u]\\
	&& 0 \arrow[u] & 0\arrow[u] 
	\end{tikzcd}
	\]
	where the rows and columns are exact, and the top right square is a pullback square. For every subgroup $H$ of $G$ we have $H^1(H,\Z[G\times G])=H^1(H,R)=0$, hence from the diagram $H^1(H,Q)=0$, that is, $Q$ is coflasque. It follows that (\ref{cofl2}) is a coflasque resolution of $M$.

	\section{The map $\alpha_N$}\label{sec4}
	We maintain the notation of the previous section. The purpose of the next two sections is the proof of \Cref{condition}, which gives a sufficient condition for $\Phi(G,M)\neq 0$. In order to prove \Cref{condition}, we must understand the map $\alpha_Q$. In this section we study $\alpha_N$, and in the next section we use the acquired knowledge to derive information on $\alpha_Q$. 
	
	\begin{lemma}\label{exterior2seq}
		Let \[0\to \Z \xrightarrow{\phi} X\xrightarrow{\psi} Y\to 0\] be a short exact sequence of $G$-lattices, and let $\sigma:=\phi(1)$. Then we have a short exact sequence of $G$-lattices \[0\to Y\xrightarrow{\eta} \ext^2(X)\xrightarrow{\wedge^2\psi}  \ext^2(Y)\to 0.\] Here, for every $y\in Y$, we set $\eta(y):=x\wedge \sigma$, where $x\in X$ is any element satisfying $\psi(x)=y$.
	\end{lemma}
	
	\begin{proof}
		We first show that $\eta$ is well defined. Let $y\in Y$, and let $x,x'\in X$ be such that $\psi(x)=\psi(x')=y$. Then $\psi(x-x')=0$, hence $x-x'=m\sigma$ for some integer $m$. But then \[x\wedge\sigma=(x'+m\sigma)\wedge\sigma=x'\wedge\sigma+ m\sigma\wedge\sigma=x'\wedge\sigma.\]
		It is clear that $\eta$ is a homomorphism of $G$-lattices, that $\wedge^2\psi$ is surjective, and that $(\wedge^2\psi) \circ \eta=0$. 
		
		Since $Y$ is torsion-free, we may complete $\sigma$ to a $\Z$-basis $x_1,\dots,x_n,\sigma$ of $X$. Then $\psi(x_1),\dots, \psi(x_n)$ is a basis of $Y$. An arbitrary element $z$ of $\ext^2(X)$ may be uniquely written as \[z=\sum_{1\leq i<j\leq n} a_{ij}x_i\wedge x_j+\sum_{i=1}^nb_ix_i\wedge \sigma\] for suitable integers $a_{ij},b_i$. In particular, $x_1\wedge\sigma,\dots,x_n\wedge\sigma$ are linearly independent. Since $\eta(\psi(x_i))=x_i\wedge \sigma$ for every $i$, we deduce that $\eta$ is injective. Moreover, the free $\Z$-module $\ext^2 (Y)$ has a basis consisting of $\psi(x_i)\wedge\psi(x_j)$, for $0\leq i<j\leq n$. We conclude that $z$ belongs to $\on{Ker}\wedge^2\psi$ if and only if $a_{ij}=0$ for every $i,j$, that is, if and only if $z$ belongs to $\on{Im}\eta$. 
	\end{proof}
	
	Let $C_2=\ang{\tau}$ be the cyclic group of order $2$, and let $\Z^-$ be the $C_2$-lattice of rank $1$ on which $\tau$ acts by $-\on{Id}$. If $g\in G$ is an element of order $2$, we denote by $\on{Ind}_{\ang{g}}^G(\Z^{-})$ the $G$-lattice induced by the $\ang{g}$-lattice $\Z^{-}$.
	
	\begin{lemma}\label{describe-n}
		Let $G$ be a finite group of order $n$, let $g_1,\dots,g_n$ be the elements of $G$, and let $\gamma:=\sum_{i=1}^n g_i\in \Z[G]$. If $m\in \Z[G]^{\oplus 2}$, we let $\cl{m}\in N$ be its image under the homomorphism $\Z[G]^{\oplus 2}\to N$ of (\ref{cofl0b}).
		\begin{enumerate}[label=(\alph*)]
			\item We have a short exact sequence
			\begin{equation}\label{n}0\to N \xrightarrow{\eta} \ext^2 (\Z[G]^{\oplus 2})\to \ext^2 N\to 0,\end{equation} where for every $m\in \Z[G]^{\oplus 2}$ we have $\eta(\cl{m})=m\wedge (\gamma,-\gamma)$.	
			\item We have \[\ext^2 (\Z[G])\simeq \Z[G]^{\oplus r}\oplus \bigoplus_{g^2=e,g\neq e}\on{Ind}_{\ang{g}}^G(\Z^{-})\] for some $r\geq 0$.
			\item For every $i\geq 1$, we have \[H^i(G,\ext^2 (\Z[G]^{\oplus 2}))\simeq \bigoplus_{g^2=e,g\neq e} H^{i+1}(\ang{g},\Z)^{\oplus 2}.\]
			\item The coboundary $\partial:H^1(G,\ext^2 N)\to H^2(G,N)$ associated to (\ref{n}) is surjective.
		\end{enumerate}
	\end{lemma}
	
	\begin{proof}
		(a) This follows from \Cref{exterior2seq}, applied to (\ref{cofl0b}). 
		
		(b) Recall that $\set{g_i\wedge g_j:i<j}$ is a basis of $\ext^2(\Z[G])$. It follows that a subset of $\set{g\wedge h:g,h\in G, g\neq h}$ is linearly dependent if and only if it contains a subset of the form $\set{g\wedge h, h\wedge g}$ for some distinct elements $g,h$ of $G$.
		
		Let $e\in G$ be the identity element. We may write \[G=\set{e}\amalg S_1\amalg S_2\amalg S_2^{-1},\] where $S_1=\set{g\in G\setminus\set{e}: g^2=e}$, and for every $g\in G$ such that $g^2\neq e$, exactly one of $g,g^{-1}$ belongs to $S_2$. For every $g\in G\setminus\set{e}$, let $M_g$ be the $G$-sublattice of $\ext^2(\Z[G])$  generated by $g\wedge e$, that is, the $\Z$-submodule generated by $\set{hg\wedge h:h\in G}$. 
		
		Assume first that $g\in S_2$. We have a $G$-homomorphism $f_g:\Z[G]\to M_g$ given by sending $e\mapsto g\wedge e$. The homomorphism $f_g$ is surjective because $M_g$ is generated by $g\wedge e$ as a $G$-lattice. The $G$-orbit of $g\wedge e$ has $n$ elements, and since $g^2\neq e$, it does not contain $e\wedge g$, and so it is a linear independent set of $n$ elements. It follows that $M_g\simeq \Z[G]$ if $g\in S_2$.
		
		Assume now that $g\in S_1$ (so $n$ is even). Then \[g(g\wedge e)=g^2\wedge g=e\wedge g=-g\wedge e.\] Let $\on{Res}^G_{\ang{g}}(M_g)$ be the restriction of $M_g$ to $\ang{g}$. The previous calculation shows that sending $1\in \Z^{-}$ to $g\wedge e$ gives a well-defined homomorphism of $\ang{g}$-modules $\Z^{-}\to \on{Res}^G_{\ang{g}}(M_g)$. By definition, we have the identifications $\on{Ind}_{\ang{g}}^G(\Z^{-})\simeq \Z^-\otimes_{\Z[\ang{g}]}\Z[G]$ and $\on{Res}_{\ang{g}}^G(M_g)\simeq \on{Hom}_{\ang{g}}(\Z[G],M_g)$, hence the tensor-hom adjunction yields a canonical isomorphism
		\[\on{Hom}_G(\on{Ind}_{\ang{g}}^G(\Z^{-}),M_g)\simeq \on{Hom}_{\ang{g}}(\Z^-,\on{Res}_{\ang{g}}^G(M_g)).\] We obtain an induced $G$-homomorphism $f_g:\on{Ind}_{\ang{g}}^G(\Z^{-})\to M_g$. As $M_g$ is generated by $g\wedge e$ as a $G$-lattice, the homomorphism $f_g$ is surjective. It is clear that $\on{Ind}_{\ang{g}}^G(\Z^{-})$ has rank $n/2$. The orbit of $g\wedge e$ has $n$ elements, and by the above calculation it contains $e\wedge g$, hence it is closed under taking the opposite. Thus $M_g$ also has rank $n/2$, hence $M_g\simeq \on{Ind}_{\ang{g}}^G(\Z^{-})$ if $g\in S_1$.
		
		To prove (b), it is thus enough to establish a $G$-equivariant direct sum decomposition
		\begin{equation}\label{wedge-dec}
		\ext^2(\Z[G])=\oplus_{g\in S_1\cup S_2}M_g.
		\end{equation}
		As an abelian group, $\ext^2(\Z[G])$ is generated by $\set{g'\wedge g:g',g\in G, g\neq g'}$. Since $g^{-1}(g'\wedge g)=g^{-1}g'\wedge e$, we see that $\ext^2(\Z[G])$ is generated by $\set{g\wedge e: g\in G\setminus\set{e}}$ as a $G$-lattice. If $g\in S_2^{-1}$, then $g^{-1}\in S_2$ and $g^{-1}(g\wedge e)=e\wedge g^{-1}=-g^{-1}\wedge e$, hence $\ext^2(\Z[G])$ is generated by $\set{g\wedge e: g\in S_1\cup S_2}$ as a $G$-lattice. In other words, $\ext^2(\Z[G])$ is the sum of the $M_g$ for $g\in S_1\cup S_2$.
		
		We conclude by a rank computation. Note that $|S_1|+2|S_2|=n-1$. We know that $M_g$ has rank $n/2$ if $g\in S_1$, and rank $n$ if $g\in S_2$. Therefore, the right hand side of (\ref{wedge-dec}) has rank $|S_1|\cdot (n/2)+|S_2|\cdot n=n(n-1)/2$. This is equal to the rank of $\ext^2(\Z[G])$, hence (\ref{wedge-dec}) holds and the proof of (b) is complete.
		
		(c) We have a short exact sequence of $C_2$-lattices \begin{equation}\label{c2}0\to \Z\to \Z[C_2]\to \Z^{-}\to 0.\end{equation} We have $H^i(C_2,\Z[C_2])=0$ for every $i\geq 1$, hence the cohomology long exact sequence associated to (\ref{c2}) gives $H^i(C_2,\Z^-)\simeq H^{i+1}(C_2,\Z)$ for every $i\geq 1$. 
		We have a $G$-equivariant decomposition \[\ext^2(\Z[G]^{\oplus 2})=\ext^2(\Z[G])\oplus\ext^2(\Z[G])\oplus (\Z[G]\otimes \Z[G]).\] 
		The conclusion now follows from (b) and the fact that $H^i(G,\Z[G])=0$ for every $i\geq 1$.
		
		(d) The group $H^3(C_2,\Z)$ is trivial. By (c), we deduce that $H^2(G,\ext^2 (\Z[G]^{\oplus 2}))$ is also trivial. The cohomology long exact sequence associated to (\ref{n}) gives \[H^1(G,\ext^2 (N))\xrightarrow{\partial} H^2(G,N)\to H^2(G,\ext^2 (\Z[G]^{\oplus 2}))=0,\] proving the surjectivity of $\partial$.
	\end{proof}
	
	\begin{lemma}\label{square}
		There exists a homomorphism $h:H^1(G,\ext^2(N))\to H^1(G,\ext^2(N))$ making the following diagram commute:	
		\begin{equation}
		\begin{tikzcd}
		H^1(G,\ext^2(N))\arrow[r,"\partial"] \arrow[d,"h"] & H^2(G,N)\arrow[d,"\pi_2"]\\
		H^1(G,\ext^2(N)) \arrow[r,"\alpha_N"] & H^2(G,N/2).	
		\end{tikzcd}
		\end{equation}
		Here $\partial$ is the connecting homomorphism of (\ref{n}), and $\alpha_N$ is the connecting homomorphism of (\ref{exterior}) with $L=N$. 
	\end{lemma}
	
	\begin{proof}
		As in \Cref{describe-n}, if $m$ is an element of $\Z[G]^{\oplus 2}$, we denote by $\cl{m}\in N$ its image under the homomorphism $\Z[G]^{\oplus 2}\to N$ of (\ref{cofl0b}). Let $n$ be the order of $G$, let $g_1,\dots,g_n$ be the elements of $G$, and denote by $(g_1,0),\dots, (g_n,0),(0,g_1),\dots,(0,g_n)$ the canonical permutation basis of $\Z[G]^{\oplus 2}$. By \Cref{setup}(a), this choice of basis yields a section
		\begin{align*}
		s:\ext^2(\Z[G]^{\oplus 2})&\to\Gamma^2(\Z[G]^{\oplus 2})\\
		(g_i,0)\wedge (g_j,0)&\mapsto \cl{(g_i,0)}\star\cl{(g_j,0)}, \quad i<j, \\
		(g_i,0)\wedge (0,g_j)&\mapsto \cl{(g_i,0)}\star\cl{(0,g_j)},\quad i\neq j, \\
		(0,g_i)\wedge (0,g_j)&\mapsto \cl{(0,g_i)}\star\cl{(0,g_j)},\quad i<j.
		\end{align*}
		We let \[f:\ext^2(\Z[G]^{\oplus 2})\xrightarrow{s} \Gamma^2(\Z[G]^{\oplus 2})\to \Gamma^2(N),\] where the second homomorphism is induced by (\ref{cofl0b}).
		Let $\eta:N\hookrightarrow \ext^2(\Z[G]^{\oplus 2})$ be the injection of \Cref{describe-n}(a), and let $\iota: N/2\hookrightarrow \Gamma^2(N)$ be the injective homomorphism of (\ref{exterior}) with $L=N$. We claim that the square
		\begin{equation}\label{surprising-comm}
		\begin{tikzcd}
		N \arrow[r,hook,"\eta"] \arrow[d,"\pi_2"] & \ext^2(\Z[G]^{\oplus 2})\arrow[d,"f"] \\
		N/2 \arrow[r,hook, "\iota"] & \Gamma^2(N) 
		\end{tikzcd}
		\end{equation} commutes. It is enough to verify the commutativity on the $\cl{(g_i,0)}$ and the $\cl{(0,g_i)}$. By \Cref{describe-n}(a), we have \begin{align*}\eta(\cl{(g_i,0)})=&(g_i,0)\wedge(\gamma,-\gamma)\\
		=&\sum_{j\neq i}(g_i,0)\wedge(g_j,0)-\sum_{j=1}^n(g_i,0)\wedge (0,g_j).\end{align*} 
		From the definition of $s$ and \Cref{ordering}, we see that  \begin{align*}s(\eta(\cl{(g_i,0)}))=&\sum_{j\neq i}(g_i,0)\star (g_j,0)-\sum_{j=1}^n(g_i,0)\star (0,g_j)\\
		=&(g_i,0)\star(\gamma,-\gamma)-(g_i,0)\star (g_i,0)\\
		=&(g_i,0)\star(\gamma,-\gamma)+(g_i,0)\star (g_i,0).
		\end{align*}
		Since $\cl{(\gamma,-\gamma)}=0$, we have \[f(\eta(\cl{(g_i,0)}))=\cl{(g_i,0)}\star\cl{(\gamma,-\gamma)}+\cl{(g_i,0)}\star \cl{(g_i,0)}=\cl{(g_i,0)}\star \cl{(g_i,0)}=\iota(\pi_2(\cl{(g_i,0)})).\] The proof of the commutativity of (\ref{surprising-comm}) on the $\cl{(0,g_i)}$ follows by symmetry. Therefore (\ref{surprising-comm}) commutes, and so there exists $\cl{f}:\ext^2(N)\to \ext^2(N)$ making the diagram
		\begin{equation}
		\begin{tikzcd}
		0 \arrow[r] & N\arrow[r,"\eta"] \arrow[d,"\pi_2"] & \ext^2(\Z[G]^{\oplus 2}) \arrow[d,"f"]\arrow[r] & \ext^2(N) \arrow[r] \arrow[d,dashed,"\cl{f}"] & 0\\
		0 \arrow[r] & N/2 \arrow[r,"\iota"] & \Gamma^2(N) \arrow[r] & \ext^2(N)\arrow[r] & 0.	
		\end{tikzcd}
		\end{equation}
		commute.
		Here the exact sequence at the top is (\ref{n}), and the exact sequence at the bottom is (\ref{exterior}) for $L=N$. Passing to group cohomology, we obtain the conclusion (the homomorphism $h$ is induced from $\cl{f}$).
	\end{proof}
	
	\section{Reduction to group cohomology with constant coefficients}\label{sec5}
	
	Recall that if $S$ is a permutation lattice, by \Cref{setup}(a) the choice of a permutation basis $x_1,\dots,x_n$ for $S$ defines a splitting of the sequence
	\[0\to S/2\to \Gamma^2(S)\to \ext^2(S)\to 0.\] More precisely, we get a section $\ext^2(S)\to \Gamma^2(S)$ by sending $x_i\wedge x_j\mapsto x_i*x_j$ for every $i<j$, and a retraction $\Gamma^2(S)\to S/2$ by sending $x_i*x_j\mapsto 0$ if $i<j$, and $x_i*x_i\mapsto x_i+ 2S$.
	
	Applying this to $S=\Z[G\times G]$, with the canonical basis $\set{(g_i,g_j)}_{i,j}$, we obtain a homomorphism
	\[\phi_N:\Gamma^2(N)\to \Gamma^2(\Z[G\times G])\to \Z[G\times G]/2,\] where the map on the left is induced by (\ref{cofl0a}).
	
	We now fix a permutation basis $\set{x_h}$ of the $G$-lattice $P$ of (\ref{cofl-p}). Using the canonical basis  $\set{(g_i,g_j)}_{i,j}$ of $\Z[G\times G]$, this extends to a permutation basis of $\Z[G\times G]\oplus P$. We obtain a homomorphism 
	\[\phi_Q:\Gamma^2(Q)\to \Gamma^2(\Z[G\times G]\oplus P)\to (\Z[G\times G]\oplus P)/2,\] where $Q$ was defined in (\ref{cofl2}), and where the homomorphism on the left is induced by the map $(\iota_1,\iota_2)$ of (\ref{cofl2}). We stress that $\phi_N$ and $\phi_Q$ depend on the choice of a permutation basis for $\Z[G\times G]$ and $P$. It may be helpful for the reader to note that $\phi_N$ and $\phi_Q$ do not depend on the choice of ordering of the bases; see \Cref{ordering}.
	
	\begin{lemma}\label{big}
		We have a commutative diagram with exact rows	
		\begin{equation*}
		\begin{tikzcd}
		0\arrow[r] & N/2 \arrow[r] \arrow[bend right=36, ddd, equal] \arrow[d,hook] & \Gamma^2(N) \arrow[r] \arrow[bend right=36, ddd,swap, "\phi_N"] \arrow[d,hook] & \ext^2(N) \arrow[r]\arrow[d,hook] \arrow[bend right=36, ddd] & 0\\
		0\arrow[r] & Q/2 \arrow[r] \arrow[d,equal] & \Gamma^2(Q) \arrow[r] \arrow[d,"\phi_Q"] & \ext^2(Q) \arrow[r] \arrow[d,"k"] & 0\\  
		0\arrow[r] & Q/2 \arrow[r]  & (\Z[G\times G]\oplus P)/2 \arrow[r]  & M/2 \arrow[d,equal] \arrow[r] & 0\\ 
		0\arrow[r] & N/2 \arrow[r] \arrow[u,hook] & \Z[G\times G]/2 \arrow[r] \arrow[u,hook] & M/2 \arrow[r] & 0. 
		\end{tikzcd}
		\end{equation*}
		Here, the first two rows are (\ref{exterior}) for $L=N,Q$. The third and fourth rows are the reductions modulo $2$ of (\ref{cofl2}) and (\ref{cofl0a}), respectively.
	\end{lemma}
	
	\begin{proof}
		By the definition of $\phi_N$ and $\phi_Q$, we have a commutative diagram
		\begin{equation*}
		\begin{tikzcd}
		\Gamma^2(N)\arrow[r] \arrow[d,hook] & \Gamma^2(\Z[G\times G]) \arrow[r] \arrow[d,hook]  & \Z[G\times G]/2 \arrow[d,hook]\\
		\Gamma^2(Q)\arrow[r] & \Gamma^2(\Z[G\times G]\oplus P) \arrow[r] & (\Z[G\times G]\oplus P)/2,
		\end{tikzcd}
		\end{equation*}
		where the composition of the top horizontal homomorphisms is $\phi_N$, and the composition of the bottom horizontal homomorphisms is $\phi_Q$. The commutativity of the square on the left follows from the functoriality of $\Gamma^2(-)$. The square on the right commutes because the choice of permutation basis that we have made respects the decomposition $\Z[G\times G]\oplus P$. We deduce that the diagram
		\begin{equation*}
		\begin{tikzcd}
		\Gamma^2(N)\arrow[r,"\phi_N"] \arrow[d,hook] &  \Z[G\times G]/2 \arrow[d,hook]\\
		\Gamma^2(Q)\arrow[r,"\phi_Q"] & (\Z[G\times G]\oplus P)/2
		\end{tikzcd}
		\end{equation*}
		is commutative. We now show that the four squares on the left side of the diagram commute. The commutativity of
		\begin{equation*}
		\begin{tikzcd}
		N/2 \arrow[r,hook] \arrow[d,hook] & \Gamma^2(N) \arrow[d,hook] \\
		Q/2 \arrow[r,hook] & \Gamma^2(Q)  
		\end{tikzcd}	\qquad\qquad
		\begin{tikzcd}
		N/2 \arrow[r,hook] \arrow[d,hook] &   \Z[G\times G]/2  \arrow[d,hook] \\
		Q/2 \arrow[r,hook] & (\Z[G\times G]\oplus P)/2
		\end{tikzcd} 
		\end{equation*}
		is clear. 
		We have a commutative diagram
		\[
		\begin{tikzcd}
		Q/2 \arrow[r,hook] \arrow[d,hook] & \Gamma^2(Q)\arrow[d,hook]  \arrow[dr,"\phi_Q"]  \\
		(\Z[G\times G]\oplus P)/2 \arrow[r,hook] & \Gamma^2(\Z[G\times G]\oplus P)\arrow[r] & (\Z[G\times G]\oplus P)/2,
		\end{tikzcd}
		\]
		where the square comes from the functoriality of (\ref{exterior}), and the triangle from the definition of $\phi_Q$. By construction, the composition of the horizontal homomorphisms at the bottom is the identity.
		Therefore, the square
		\begin{equation*}
		\begin{tikzcd}
		Q/2 \arrow[r,hook] \arrow[d,equal] & \Gamma^2(Q) \arrow[d,"\phi_Q"] \\
		Q/2 \arrow[r,hook] & (\Z[G\times G]\oplus P)/2   
		\end{tikzcd} 
		\end{equation*}
		appearing in the diagram of the lemma commutes. A similar reasoning yields the commutativity of
		\begin{equation*}
		\begin{tikzcd}
		N/2 \arrow[r,hook] \arrow[d,equal] & \Gamma^2(N) \arrow[d,"\phi_N"] \\
		N/2 \arrow[r,hook] & \Z[G\times G]/2.   
		\end{tikzcd} 
		\end{equation*}
		The existence and commutativity of the right side of the diagram of the lemma now follow from the universal property of cokernels.
	\end{proof}

	For every $i\geq 0$, let $\pi_2:H^i(G,\Z)\to H^i(G,\Z/2)$ be the homomorphism of reduction modulo $2$, and let $\on{Sq}^1:H^i(G,\Z/2)\to H^{i+1}(G,\Z/2)$ be the first Steenrod square, that is, the Bockstein homomorphism for the sequence $0\to \Z/2\to \Z/4\to \Z/2\to 0$. We also denote by $\beta:H^i(G,\Z/2)\to H^{i+1}(G,\Z)$ the Bockstein homomorphism for $0\to \Z\to \Z\to \Z/2\to 0$. It is well known and easy to show that $\on{Sq}^1=\pi_2\circ\beta$.
	
	For every subgroup $H$, let $\on{Cor}_H^G:H^*(H,\Z/2)\to H^*(G,\Z/2)$ denote the corestriction homomorphism, and let $\tau_H$ be the composition
	\begin{align}\label{cup}\tau_H:H^1(H,\Z/2)\otimes H^2(H,\Z)\xrightarrow{\on{Id}\otimes \pi_2} H^1(H,\Z/2)\otimes H^2(H,\Z/2)\xrightarrow{\cup} \\ \xrightarrow{\cup} H^3(H,\Z/2)\xrightarrow{\on{Cor}_H^G} H^3(G,\Z/2)\nonumber.\end{align}
	
	\begin{defin}\label{defin-vg}
		Let $G$ be a finite group. We define $V_G$ as the subgroup of $H^3(G,\Z/2)$ generated by the union of the $\on{Im}\tau_H$, where $H$ varies among all subgroups of $G$.
	\end{defin}
	
	\begin{prop}\label{condition}
		\begin{enumerate}[label=(\alph*)]
			\item Assume that $\on{Im}(\pi_2:H^3(G,\Z)\to H^3(G,\Z/2))$ is not contained in $V_G$. Then $\Phi(G,M)\neq 0$.
			\item Assume that $\on{Im}(\on{Sq}^1:H^2(G,\Z/2)\to H^3(G,\Z/2))$ is not contained in $V_G$. Then $\Phi(G,M)\neq 0$.
		\end{enumerate}
	\end{prop}
	
	\begin{proof}
		Following the notation of Merkurjev \cite[\S 3]{merkurjev2019pairing}, using (\ref{cofl2}) we define \begin{align*}H^1(G,M/2)^{(1)}:=&\on{Im}(H^1(G,(\Z[G\times G]\oplus P)/2)\to H^1(G,M/2))\\ =&\on{Ker}(H^1(G,M/2)\to H^2(G,Q/2)).\end{align*} 
		By \cite[Lemma 3.1]{merkurjev2019pairing}, $H^1(G,M/2)^{(1)}$ is generated by the images of the compositions
		\[\sigma_H:H^1(H,\Z/2)\otimes M^H\xrightarrow{\cup} H^1(H,M/2)\xrightarrow{\on{Cor}_H^G} H^1(G,M/2),\] where $H$ ranges among all subgroups of $G$. 
		For every subgroup $H$ of $G$ and every $i\geq 0$, we denote by \begin{align*}d_1:H^i(H,M)\to H^{i+1}(H,N),& \qquad d_2:H^i(H,N)\to H^{i+1}(H,\Z),\\ \delta_1:H^i(H,M/2)\to H^{i+1}(H,N/2),& \qquad \delta_2:H^i(H,N/2)\to H^{i+1}(H,\Z/2)\end{align*}
		the connecting homomorphisms associated to (\ref{cofl0a}), (\ref{cofl0b}) and their reduction modulo $2$. Recall that free $\Z[G]$-modules and free $(\Z/2)[G]$-modules have trivial cohomology in positive degrees. It follows that $d_1,d_2,\delta_1,\delta_2$ are surjective in all non-negative degrees, and isomorphisms in positive degrees. 
		
		Let $H$ be a subgroup of $G$. By a double application of \cite[Proposition 3.4.8]{gille2017central}, first to (\ref{cofl0a}) and then to (\ref{cofl0b}), the diagram
		\begin{equation*}
		\begin{tikzcd}
		H^1(H,\Z/2)\otimes M^H \arrow[r,"\cup"] \arrow[d,->>,"\on{Id}\otimes (d_2\circ d_1)"] & H^1(H,M/2) \arrow[d,"\wr\: \delta_2\circ\delta_1"]\\
		H^1(H,\Z/2)\otimes H^2(H,\Z) \arrow[r,"\cup"] & H^3(H,\Z/2)	
		\end{tikzcd}
		\end{equation*}
		commutes. Here, the horizontal homomorphism at the bottom is (\ref{cup}). Since connecting homomorphisms commute with corestrictions, for every subgroup $H$ of $G$ we obtain the following commutative square:
		\begin{equation*}
		\begin{tikzcd}
		H^1(H,\Z/2)\otimes M^H \arrow[r,"\sigma_H"] \arrow[d,->>,"\on{Id}\otimes (d_2\circ d_1)"] & H^1(G,M/2) \arrow[d,"\wr\: \delta_2\circ\delta_1"]\\
		H^1(H,\Z/2)\otimes H^2(H,\Z) \arrow[r,"\tau_H"] & H^3(G,\Z/2).	
		\end{tikzcd}
		\end{equation*}
		We conclude that
		\begin{equation}\label{vg}
		(\delta_2\circ\delta_1)(H^1(G,M/2)^{(1)})=V_G.
		\end{equation} 	
		
		(a) Consider the commutative diagram
		\begin{equation*}
		\begin{tikzcd}
		& H^1(G,\ext^2(N)) \arrow[d,->>,"\partial"] \arrow[dd,bend right=54,near end, swap,"\alpha_N\circ h"] \\
		H^1(G,M) \arrow[r,"d_1"] \arrow[d,"\pi_2"] & H^2(G,N) \arrow[r,"d_2"] \arrow[d,"\pi_2"] & H^3(G,\Z) \arrow[d,"\pi_2"]\\
		H^1(G,M/2) \arrow[r,"\delta_1"] & H^2(G,N/2)\arrow[r,"\delta_2"] & H^3(G,\Z/2)	
		\end{tikzcd}
		\end{equation*}
		where the triangle in the middle comes from \Cref{square}, the horizontal maps are isomorphisms and $\partial$ is the surjective homomorphism of \Cref{describe-n}(d). 
		We claim that there exists $x\in H^1(G,\ext^2(N))$ such that 
		\begin{equation}\label{exists}
		\delta_1^{-1}(\alpha_N(x))\in H^1(G,M/2) \setminus H^1(G,M/2)^{(1)}.
		\end{equation}
		By assumption, there exists $a\in H^3(G,\Z)$ such that $\pi_2(a)\in H^3(G,\Z/2)\setminus V_G$. We let $b:=(d_2\circ d_1)^{-1}(a)\in H^1(G,M)$. By (\ref{vg}), $\pi_2(b)\in H^1(G,M/2)\setminus H^1(G,M/2)^{(1)}$. Since $\partial$ is surjective, there exists $c\in H^1(G,\ext^2(N))$ such that $\partial(c)=d_1(b)$. If we let $x:=h(c)\in H^1(G,\ext^2(N))$, then \[\delta_1^{-1}(\alpha_N(x))=\delta_1^{-1}(\pi_2(\partial(c)))=\pi_2(d_1^{-1}(\partial(c)))=\pi_2(b),\] hence $x$ satisfies (\ref{exists}). 
		
		Passing to group cohomology in \Cref{big}, we obtain a commutative diagram
		\begin{equation*}
		\begin{tikzcd}
		H^1(G,\ext^2(N)) \arrow[r] \arrow[d,"\alpha_N"] \arrow[rrr,bend left=9]  & H^1(G,\ext^2(Q)) \arrow[r,"k^*"] \arrow[d,"\alpha_Q"] & H^1(G,M/2) \arrow[r,equal] \arrow[d] & H^1(G,M/2) \arrow[d,"\wr\: \delta_1"]\\
		H^2(G,N/2) \arrow[r] \arrow[rrr,equal,bend right=9] & H^2(G,Q/2) \arrow[r,equal] & H^2(G,Q/2) & \arrow[l] H^2(G,N/2).	
		\end{tikzcd}
		\end{equation*}
		Let now $x\in H^1(G,\ext^2(N))$ be such that (\ref{exists}) holds for $x$, and let $y\in H^1(G,\ext^2(Q))$ be the image of $x$. From the commutativity of the outer square of the diagram above, we see that the composition of the homomorphisms of the top row is $\delta_1^{-1}\circ \alpha_N$. Therefore \[k^*(y)=\delta_1^{-1}(\alpha_N(x))\in H^1(G,M/2)\setminus H^1(G,M/2)^{(1)}.\] By definition of $H^1(G,M/2)^{(1)}$, this means that $k^*(y)$ does not map to zero in $H^2(G,Q/2)$. It follows from the commutativity of the middle square of the diagram that $\alpha_Q(y)$ is a non-zero element of $H^2(G,Q/2)$. This shows that $\on{Im}\alpha_Q\neq 0$, that is, $\Phi(G,M)\neq 0$.
		
		(b) Since $\on{Sq}^1=\pi_2\circ\beta$, this follows immediately from (a).	
	\end{proof}
	
	\begin{rmk}
		The hypotheses of (a) and (b) are equivalent if $H^3(G,\Z)$ is $2$-torsion. In general we only have that the hypothesis of (b) implies that of (a). However, the hypothesis of (b) is easier to check: if $H^3(G,\Z)$ has exponent $2^m$, to check the hypothesis of (a) one needs to know the degree $3$ differentials of the first $m$ pages of the Bockstein spectral sequence associated to the short exact sequence of $G$-modules $0\to \Z\to \Z\to \Z/2\to 0$. 
		
		If $G$ is abelian or has order $\leq 32$, we have checked that the hypotheses of (a) and (b) are not satisfied.
	\end{rmk}
	
	\section{Proof of Theorem 1.2}\label{sec6}
	
	Let $G$ be a $2$-Sylow subgroup of the Suzuki group $\on{Sz}(8)$. We start by giving an explicit description of $G$. Let $\F_8$ be the field of $8$ elements, let $\F_8^+$ be its underlying additive group, and let $\theta:\F_8\to \F_8$ be the field automorphism given by $\theta(a)=a^4$, so that $\theta(\theta(a))=a^2$. For a pair of elements $a,b\in \F_8$, let $S(a,b)$ be the following lower-triangular matrix with entries in $\F_8$:
	\[S(a,b):=
	\begin{pmatrix}
	1 & 0 & 0 & 0 \\
	a& 1 & 0 & 0 \\
	b & \theta(a) & 1 & 0 \\
	a^2\theta(a)+ab+\theta(b) & a\theta(a)+b & a & 1\\
	\end{pmatrix}
	\]
	Then $G$ can be identified with the set of matrices $S(a,b)$; see \cite[Chapter XI, \S 3]{huppert2012finite}. A matrix calculation shows that \[S(a,b)S(c,d)=S(a+c,\theta(a)c+b+d).\]
	It follows that we have a short exact sequence
	\[1\to Z\to G\xrightarrow{\pi}\cl{G}\to 1,\] where $\cl{G}:=\F_8^+\simeq (\Z/2)^3$, $\pi(S(a,b)):=a$, and $Z\simeq (\Z/2)^3$ is the center of $G$. The subgroup $Z$ coincides with the derived subgroup of $G$, and so $\cl{G}$ is the abelianization of $G$; see \cite[Chapter 11, Lemma 3.1]{huppert2012finite}.
	
	If $R$ is a ring, $A^*=\oplus_{n\geq 0} A^n$ is a graded $R$-algebra, and $i$ is a non-negative integer, we denote by $A^{\leq i}$ the quotient of $A^*$ by the ideal $\oplus_{n>i} A^n$.
	The ring $H^*(G,\Z/2)$ is extremely complicated. It may be found, together with restriction and corestriction homomorphisms, in the book \cite[\#153(64) p. 566]{carlson2003cohomology}. It can also be obtained using a computer algebra software such as Magma or GAP, where $G$ is SmallGroup(64,82). For our purposes, we are only interested in degrees $\leq 3$. We have:
	\[H^{\leq 3}(G,\Z/2)\simeq (\Z/2)[z_1,y_1,x_1,w_2,v_2,u_3,t_3,s_3,r_3,q_3,p_3]^{\leq 3}/I.\] Here the indices denote the degrees of the variables, and will be dropped in the future. The ideal $I$ is defined by \[I:=(z^2 + yx, zy + zx + x^2, zx + y^2 + yx, zx^2, yx^2, x^3, zw+xv, zv+yw, yv+xw+xv).\] 
	
	\begin{lemma}\label{sqw}
	Let $V_G$ be as in \Cref{defin-vg}. The class $\on{Sq}^1(w)$ does not belong to $V_G$.
	\end{lemma}
	
	\begin{proof}
		Let \[W_G:=\on{Im}(H^1(G,\Z/2)^{\otimes 3}\xrightarrow{-\cup-\cup-} H^3(G,\Z/2))\] be the image of the triple cup product map, and let $V_G$ be as in \Cref{defin-vg}. We claim that $V_G\c W_G$. Since $H^1(G,\Z)=0$ and \[H^2(G,\Z)\simeq H^1(G,\Q/\Z)\simeq \on{Hom}(G,\Q/\Z)\simeq \on{Hom}(\cl{G},\Q/\Z)\simeq (\Z/2)^3\] is $2$-torsion, the Bockstein homomorphism $\beta:H^1(G,\Z/2)\to H^2(G,\Z)$ is an isomorphism. It follows that $H^2(G,\Z)$ is generated by $\beta(z),\beta(y),\beta(x)$, and so the image of $\pi_2:H^2(G,\Z)\to H^2(G,\Z/2)$ is generated by $z^2,y^2,x^2$ (recall that $\on{Sq}^1=\pi_2\circ\beta$, and that $\on{Sq}^1(c)=c^2$ when $c$ has degree $1$). It follows that $\on{Im}(\tau_G)\c W_G$, where $\tau_G$ is from (\ref{cup}).
		
		If $H$ is a subgroup of $G$ and $K$ is a subgroup of $H$, then $\on{Cor}_K^G=\on{Cor}_H^G\circ \on{Cor}_K^H$, hence to prove that $V_G\c W_G$ it is enough to show that $\on{Im}\on{Cor}_H^G\c W_G$ for every maximal proper subgroup $H$ of $G$. The ideal of $H^*(G,\Z/2)$ generated by the images of the $\on{Cor}_H^G$, where $H$ ranges over all maximal proper subgroups $H$ of $G$, can be read from \cite[\#153(64) ImTrans, p.569]{carlson2003cohomology}. In degree $\leq 3$, it coincides with the ideal generated by $zx,yx,x^2$. Therefore, in degree $3$ we have
		\[\on{Im}\on{Cor}_H^G\c \ang{zxa,yxb,x^2c: a,b,c\in H^1(G,\Z/2)}\c W_G,\] hence $V_G\c W_G$, as claimed.
		
		To finish the proof, it suffices to show that $\on{Sq}^1(w)$ does not belong to $W_G$. Let $G_1$ be the first maximal subgroup of $G$ in the list of \cite[p. 569]{carlson2003cohomology} (that is, \#1 in MaxRes). Its cohomology ring can be found in \cite[\#18(32), p. 367]{carlson2003cohomology}:
		\[H^*(G_1,\Z/2)\simeq(\Z/2)[z'_1,y'_1,x'_2,w'_2,v'_2,u'_2,t'_2]/I',\] for some ideal $I'$, a precise description of which we will not need. By \cite[MaxRes, p. 569]{carlson2003cohomology}, the restriction $\on{Res}_{G_1}^G$ sends \[z\mapsto 0,\quad y\mapsto y',\quad x\mapsto z',\quad w\mapsto v'.\] Let $G_2$ be the first maximal subgroup of $G_1$ in the list of \cite[MaxRes, p. 368]{carlson2003cohomology} (we have $G_2\simeq \Z/4\times \Z/2\times \Z/2$). The cohomology ring of $G_2$ can be read from \cite[\#2(16), p. 349]{carlson2003cohomology}:
		\[H^*(G_2,\Z/2)\simeq (\Z/2)[z''_1,y''_1,x''_1,w''_2]/((z''_1)^2).\] 
		By \cite[MaxRes \#1, p. 368]{carlson2003cohomology}, $\on{Res}^{G_1}_{G_2}$ sends \[z'\mapsto z'',\quad y'\mapsto 0,\quad v'\mapsto z''x''.\] 
		It follows that $\on{Res}^{G}_{G_2}=\on{Res}_{G_2}^{G_1}\circ \on{Res}_{G_1}^{G}$ sends \[z\mapsto 0,\quad y\mapsto 0,\quad x\mapsto z'',\quad w\mapsto z''x''.\] Let $\alpha:=z^ly^mx^n\in H^3(G,\Z/2)$, where $l,m,n\geq 0$ and $l+m+n=3$. If $l\geq 1$ or $m\geq 1$, then clearly $\on{Res}_{G_2}^G(\alpha)=0$. If $l=m=0$, then $\alpha=x^3$, and so $\on{Res}_{G_2}^G(\alpha)=(z'')^3=0$. It follows that $\on{Res}_{G_2}^G(W_G)=\set{0}$. On the other hand, since Steenrod operations commute with restriction homomorphisms, we have
		\[\on{Res}_{G_2}^G(\on{Sq}^1(w))=\on{Sq}^1(z''x'')=(z'')^2x''+z''(x'')^2=z''(x'')^2\neq 0.\] This shows that $\on{Res}_{G_2}^G(\on{Sq}^1(w))$ does not belong to $\on{Res}_{G_2}^G(W_G)=\set{0}$. We conclude that $\on{Sq}^1(w)$ does not belong to $W_G$, as desired. 
	\end{proof}
	
	\begin{proof}[Proof of \Cref{mainthm}]
		Let $G$ be a $2$-Sylow subgroup of the Suzuki group $\on{Sz}(8)$. By \Cref{sqw}, the image of $\on{Sq}^1:H^2(G,\Z/2)\to H^3(G,\Z/2)$ is not contained in $V_G$. Let $M$ be the $G$-lattice of (\ref{cofl0}). Then by \Cref{condition}(b) we have $\Phi(G,M)\neq 0$. Let $V$ be a faithful representation of $G$ over $\Q$, and set $E:=\Q(V)$ and $F:=\Q(V)^G$. The extension $E/F$ is Galois, with Galois group $G$. Let $T$ be an $F$-torus split by $E$ and whose character lattice $\hat{T}$ is isomorphic to $M$. By \Cref{merk-thm}(b), we have $\on{CH}^2(BT)_{\on{tors}}\neq 0$.
	\end{proof}
	
	\section*{Acknowledgments}
	I am thankful to Jean-Louis Colliot-Th\'el\`ene for helpful comments, and for hosting me at Universit\'e Paris-Sud (Orsay) in the Fall 2019, where this work was conducted. I am grateful to Mathieu Florence for a number of helpful conversations on the topic of this article. I thank my advisor Zinovy Reichstein and Ben Williams for comments on the exposition, and the anonymous referees for carefully reading my manuscript and providing useful suggestions.

\end{document}